\documentclass{amsart}
\usepackage{amsmath,amssymb,amsfonts}
\usepackage[mathscr]{eucal}

\usepackage{enumerate}

\theoremstyle{plain}
\newtheorem{theorem}{Theorem}[section]

\theoremstyle{definition}
\newtheorem{definition}[theorem]{Definition}
\newtheorem{example}[theorem]{Example}
\newtheorem*{thank}{Acknowledgments}

\numberwithin{equation}{section}

\input xy
\xyoption{all}

\newcommand{\Deltaop}{{\bf \Delta}^{op}}

\newcommand{\hocolim}{\text{hocolim}}
\newcommand{\holim}{\text{holim}}
\newcommand{\nerve}{\text{nerve}}

\newcommand{\we}{\text{we}}

\newcommand{\Hom}{\text{Hom}}
\newcommand{\Map}{\text{Map}}
\newcommand{\Aut}{\text{Aut}}

\newcommand{\Top}{\mathcal Top}
\newcommand{\SSets}{\mathcal{SS}ets}

\newcommand{\Sets}{\mathcal Sets}
\newcommand{\map}{\text{map}}
\newcommand{\ob}{\text{ob}}
\newcommand{\hoequiv}{\text{hoequiv}}
\newcommand{\css}{\mathcal{CSS}}

\begin{document}

\title[Homotopy limits]{Homotopy limits of model categories and more general homotopy theories}

\author[J.E. Bergner]{Julia E. Bergner}

\address{Department of Mathematics, University of California, Riverside, CA 92521}

\email{bergnerj@member.ams.org}

\date{\today}

\subjclass[2000]{Primary: 55U40; Secondary: 55U35, 18G55, 18G30,
18D20}

\keywords{model categories, complete Segal spaces, $(\infty,1)$-categories, homotopy theories, homotopy limits}

\thanks{The author was partially supported by NSF grant DMS-0805951.}

\begin{abstract}
Generalizing a definition of homotopy fiber products of model categories, we give a definition of the homotopy limit of a diagram of left Quillen functors between model categories.  As has been previously shown for homotopy fiber products, we prove that such a homotopy limit does in fact correspond to the usual homotopy limit, when we work in a more general model for homotopy theories in which they can be regarded as objects of a model category.
\end{abstract}

\maketitle

\section{Introduction}

Model categories, first defined by Quillen in \cite{quillen}, have long been a useful tool in understanding the homotopy theories of various mathematical structures.  As with the classical homotopy theory of spaces, having a model structure enables one to make various constructions, such as homotopy limits, unambiguously defined.

A more modern viewpoint might suggest regarding model categories themselves as objects of study.  In this way, one could seek to understand relationships between different model categories.  Left and right Quillen functors provide the correct kinds of maps, with Quillen equivalences our standard means of considering two model categories sufficiently alike. In this framework, one could ask questions such as what a homotopy limit or homotopy colimit of a diagram of model categories would be.  Unfortunately, there are no immediate answers to these questions because at present there is no known model structure on the category of model categories.  For the special case of the homotopy pullback of model categories, a construction has been used, for example by To\"en in \cite{toendha}.  In \cite{fiberprod} we show that it is in fact the ``correct" definition, by regarding model categories as particularly nice examples of a more general notion of ``homotopy theory."

We could, alternatively, consider a category with a specified class of weak
equivalences, or maps one would like to consider as equivalences
but which are not necessarily isomorphisms.  There may or
may not be a model structure on such a category, but the basic data of a ``homotopy theory" is present: some objects together with a means of calling two of them equivalent to one another.  This broader approach is useful in that one can use it to investigate a homotopy theory even if it does not have the more rigid structure of a model category.  While this flexibility comes at a cost, namely the lack of nice constructions arising from a model structure, we have the benefit that there is a model structure on the collection of all homotopy theories.  There are several choices of models for such homotopy theories; in this paper we work with the complete Segal space model due to the fact that it has the simplest criteria for identifying weak equivalences.

In \cite{fiberprod}, we proved that taking the homotopy pullback construction on a diagram of model categories and then passing to the world of complete Segal spaces gives a result weakly equivalent to the one we get if we translate the whole diagram into a diagram of complete Segal spaces and then take the usual homotopy pullback in this model category.  This fact thus confirms that the definition of homotopy pullback of model categories was the correct one.

In this paper, we address the more general question of defining a homotopy limit of a diagram of model categories and its analogue within the complete Segal space model structure.  We give a definition of such a homotopy limit of model categories, then establish its validity using the same approach that we did for the special case of homotopy pullbacks.

Using the idea that a homotopy limit is analogous to some kind of sheaf, the approaches of this paper can be compared to model topoi of Rezk \cite{rezktopos} and the more general $(\infty, 1)$-versions of topoi given by To\"en and Vezzosi \cite{tv} and Lurie \cite{lurie}.  A related construction of homotopy limits has also been given by Barwick \cite{barwick}.

This approach to homotopy limits of model categories lends itself to translating various results for spaces in classical homotopy theory and generalizing them to the model categorical level.  For example, in Section \ref{ex}, we define a Postnikov tower of categories of topological spaces with homotopy groups nonzero only through a given range, gradually increasing, and show that this tower converges to the category of topological spaces.

\begin{thank}
The author would like to thank Andrew Blumberg, Emmanuel Farjoun, and Philip Hackney for helpful conversations about this work.  We are also indebted to the anonymous referee for pointing out Example \ref{neweq} and for encouraging that the details of Section \ref{ex} be explained here.
\end{thank}

\section{Model categories and more general homotopy theories}

In this section we give a brief review of model categories and their relationship with the complete Segal space model for more general homotopy theories.

Recall that a \emph{model category} $\mathcal M$ is a category with three distinguished classes of morphisms: weak equivalences, fibrations, and cofibrations, satisfying five axioms \cite[3.3]{ds}.  An object $x$ in $\mathcal M$ is \emph{fibrant} if the unique map $x \rightarrow \ast$ to the terminal object is a fibration.  Dually, an object $x$ in $\mathcal M$ is \emph{cofibrant} if the unique map $\phi \rightarrow x$ from the initial object is a cofibration.

Given a model category $\mathcal M$, there is also a model structure on the category $\mathcal M^{[1]}$ whose objects of $\mathcal M^{[1]}$ are morphisms of $\mathcal M$, and whose morphisms are given by pairs of morphisms in $\mathcal M$ making the appropriate square diagram commute.  A morphism in $\mathcal M^{[1]}$ is a weak equivalence (or cofibration) if its component maps are weak equivalences (or cofibrations) in $\mathcal M$.  More generally, $\mathcal M^{[n]}$ is the category with objects strings of $n$ composable morphisms in $\mathcal M$; the model structure can be defined analogously.

One could also consider categories with weak equivalences and no additional structure, and then formally invert the weak equivalences.  This process does give a homotopy category, but it can have the weakness of having a proper class of morphisms between any two given objects.  If we are willing to accept such set-theoretic problems, then we can work in this situation; the advantage of a model structure is that it provides enough additional structure so that we can take homotopy classes of maps and hence avoid these difficulties.

To understand objects modeling such homotopy theories, we need the language of simplicial objects.  Recall that the simplicial indexing category $\Deltaop$ is defined to be the category with objects finite ordered sets $[n]=\{0 \rightarrow 1 \rightarrow \cdots \rightarrow n\}$ and morphisms the opposites of the order-preserving maps between them.  A \emph{simplicial set} is a functor
\[ K \colon \Deltaop \rightarrow \Sets. \]
We denote by $\SSets$ the category of simplicial sets, and this category has a natural model category structure equivalent to the standard model structure on topological spaces \cite[I.10]{gj}.

In this paper we also use \emph{simplicial spaces} (also called bisimplicial sets), or functors
\[ X \colon \Deltaop \rightarrow \SSets. \]
There are several model category structures on the category of bisimplicial sets.  An important one is the Reedy model structure \cite{reedy}, which is equivalent to the injective model structure, where the weak equivalences and cofibrations are defined levelwise \cite[15.8.7]{hirsch}.  Given a simplicial set $K$, we also denote by $K$ the constant simplicial space which has the simplicial set $K$ at every level.  We denote by $K^t$, or ``$K$-transposed", the constant simplicial space in the other direction, where $(K^t)_n = K_n$, where on the right-hand side $K_n$ is regarded as a discrete simplicial set.

A simplicial category, or category enriched over simplicial sets, models a homotopy theory, in a sense developed by Dwyer and Kan.  Using either of their two notions of simplicial localization, one can obtain from a category with weak equivalences a simplicial category \cite{dkfncxes}, \cite{dksimploc}, and there is a model structure $\mathcal{SC}$ on the category of all small simplicial categories \cite{simpcat}.  In this way, we obtain a ``homotopy theory of homotopy theories" in which a homotopy theory is modeled by a simplicial category.

One useful consequence of taking the simplicial category corresponding to a model category is that we can use it to describe \emph{homotopy function complexes}, or homotopy-invariant mapping spaces $\Map^h(x,y)$ between objects of a model category which is not necessarily equipped with the additional structure of a simplicial model category.  Of particular interest in this paper is the simplicial set $\Aut^h(x)$ of homotopy-invertible self-maps of an object $x$.

Rather than simplicial categories, we can consider complete Segal spaces, first introduced by Rezk \cite{rezk}, given by simplicial spaces satisfying conditions imposing a notion of composition up to homotopy.

\begin{definition} \cite[4.1]{rezk}
A \emph{Segal space} is a Reedy fibrant simplicial space $W$ such that the Segal maps
\[ \varphi_n \colon W_n \rightarrow \underbrace{W_1 \times_{W_0} \cdots \times_{W_0} W_1}_n \] are weak equivalences of simplicial sets for all $n \geq 2$.
\end{definition}

Given a Segal space $W$, we can consider its ``objects" $\ob(W)= W_{0,0}$, and, between any two objects $x$ and $y$, the ``mapping space" $\map_W(x,y)$, defined to be the homotopy fiber of the map $W_1 \rightarrow W_0 \times W_0$ given by the two face maps $W_1 \rightarrow W_0$.  The Segal condition given here tells us that a Segal space has a notion of $n$-fold composition of mapping spaces, up to homotopy.  Using this composition, we can define ``homotopy equivalences" in a natural way, and then speak of the subspace of $W_1$ whose components contain homotopy equivalences, denoted $W_{\hoequiv}$.  Notice that the degeneracy map $s_0 \colon W_0 \rightarrow W_1$ factors through $W_{\hoequiv}$.

\begin{definition} \cite[\S 6]{rezk}
A \emph{complete Segal space} is a Segal space $W$ such that the map $W_0 \rightarrow W_{\hoequiv}$ is a weak equivalence of simplicial sets.
\end{definition}

\begin{theorem} \cite[\S 7]{rezk}
There is a model category structure $\css$ on the category of simplicial spaces, obtained as a localization of the Reedy model structure such that:
\begin{enumerate}
\item the fibrant objects are the complete Segal spaces,

\item all objects are cofibrant, and

\item the weak equivalences between complete Segal spaces are levelwise weak equivalences of simplicial sets.
\end{enumerate}
\end{theorem}

The following theorem tells us that we can think of complete Segal spaces as models for homotopy theories.

\begin{theorem} \cite{thesis}
The model categories $\mathcal{SC}$ and $\css$ are Quillen equivalent.
\end{theorem}

While not the one used for this Quillen equivalence, Rezk defines a functor which we denote $L_C$ from the category of model categories and left Quillen functors to the category of simplicial spaces; given a model category $\mathcal M$, we have that
\[ L_C(\mathcal M)_n = \nerve(\we(\mathcal M^{[n]})). \]  Here, $\mathcal M^{[n]}$ is defined as above, and $\we(\mathcal M^{[n]})$ denotes the subcategory of $\mathcal M^{[n]}$ whose morphisms are the weak equivalences.  While the resulting simplicial space is not in general Reedy fibrant, and hence not a complete Segal space, Rezk proves that taking a Reedy fibrant replacement is sufficient to obtain a complete Segal space \cite[8.3]{rezk}.  For the rest of this paper we assume that the functor $L_C$ includes composition with this Reedy fibrant replacement and therefore assigns a complete Segal space to any model category.  In fact, this construction can be applied to any category with weak equivalences.

Another problem with this definition is the fact that it is only a well-defined functor on the category whose objects are model categories and whose morphisms preserve weak equivalences.  We would prefer to have such a functor defined on the category of model categories with morphisms left Quillen functors.  To obtain such a functor, we consider $\mathcal M^c$, the full subcategory of $\mathcal M$ whose objects are cofibrant.  While $\mathcal M^c$ may no longer have the structure of a model category, it is still a category with weak equivalences.  Thus, we define
\[ L_C(\mathcal M)_n = \nerve(\we((\mathcal M^c)^{[n]})). \]  Each space in this new diagram is weakly equivalent to the one given by the previous definition, and now the construction is functorial on the category of model categories with morphisms the left Quillen functors.  If one wanted to consider right Quillen functors instead, we could take the full subcategory of fibrant objects, $\mathcal M^f$, rather than $\mathcal M^c$.

With these restrictions in place, the image of a model category under this functor is weakly equivalent to the one obtained by taking its simplicial localization and then applying the Quillen equivalence \cite[6.2]{css}.  Furthermore, this resulting complete Segal space can be characterized up to weak equivalence.

Before stating the characterization, we first need some notation.  Given a simplicial
monoid $M$, we can find a classifying complex
of $M$, a simplicial set whose geometric realization is the
classifying space $BM$. A precise construction can be made for
this classifying space by the $\overline W M$ construction
\cite[V.4.4]{gj}, \cite{may}.  We will simply write
$BM$ for the classifying complex of $M$.  Later in the paper, we consider disjoint unions of simplicial monoids; in this case the classifying complex is taken in the category of simplicial categories, rather than in simplicial monoids.

\begin{theorem} \cite[7.3]{css}  \label{baut}
Let $\mathcal M$ be a model category.  For $x$ an object of $\mathcal M$ denote by $\langle x \rangle$ the weak equivalence class of $x$ in $\mathcal M$, and denote by $Aut^h(x)$ the simplicial monoid of self weak equivalences of $x$.  Similarly, let $\langle \alpha \colon x \rightarrow y \rangle$ denote the weak equivalence class of $\alpha$ in $\mathcal M^{[1]}$ and $Aut^h(\alpha)$ its respective simplicial monoid of self weak equivalences.  Up to weak equivalence in the model category $\css$, the complete Segal space $L_C(\mathcal M)$ looks like
\[ \coprod_{\langle x \rangle} BAut^h(x) \Leftarrow \coprod_{\langle \alpha \colon x \rightarrow y \rangle} BAut^h(\alpha) \Lleftarrow \cdots. \]
\end{theorem}

We should point out that the reference (Theorem 7.3 of \cite{css}) gives a characterization of the complete Segal space arising from a simplicial category, not from a model category.  However, the results of \S 6 of that same paper allow one to translate it to the theorem as stated here.


This characterization, together with the fact that weak equivalences between complete Segal spaces are levelwise weak equivalences of simplicial sets, enables us to compare complete Segal spaces with one another.

\section{Homotopy limits of model categories}

Let $\mathcal D$ be a small category.  A $\mathcal D$-shaped diagram $X$ is given by a collection $\mathcal M_\alpha$, one for each object $\alpha$ of $\mathcal D$, together with left Quillen functors $F_{\alpha, \beta}^\theta \colon \mathcal M_\alpha \rightarrow \mathcal M_\beta$ which are compatible with one another, in the sense that if $\theta \colon \alpha \rightarrow \beta$ and $\delta \colon \beta \rightarrow \gamma$ are composable maps in $\mathcal D$, then
\[ F_{\alpha, \gamma}^{\delta \theta}=F_{\beta, \gamma}^\delta \circ F_{\alpha, \beta}^\theta. \]
The superscript $\theta$ indexes different left Quillen functors between the same model categories, coming from distinct maps $\theta \colon \alpha \rightarrow \beta$.  More precisely, if one considers the (large) category $\mathcal {MC}$ of model categories with left Quillen functors between them (or some small subcategory of it), then we can just give such a describe by a functor $X \colon \mathcal D \rightarrow \mathcal {MC}$.

\begin{definition}
Let $X$ be a $\mathcal D$-shaped diagram of model categories.  Then a \emph{lax homotopy limit} for $X$, denoted $\mathcal L_\mathcal D X$, is given by a family $(x_\alpha, u^\theta_{\alpha, \beta})$ where $x_\alpha$ is an object of $\mathcal M_\alpha$ and $u^\theta_{\alpha, \beta} \colon F^\theta_{\alpha, \beta}(x_\alpha) \rightarrow x_\beta$ is a morphism in $\mathcal M_\beta$, satisfying the compatibility condition
\[ u_{\alpha, \gamma}^{\delta \theta}= u_{\beta, \gamma}^\delta \circ F_{\beta, \gamma}^\delta(u_{\alpha, \beta}^\theta). \]

The \emph{homotopy limit} for $X$, denoted $\mathcal Lim_\mathcal D X$, is the full subcategory of $\mathcal L_\mathcal DX$ whose objects satisfy the additional condition that all maps $u_{\alpha, \beta}^\theta$ are weak equivalences in their respective $\mathcal M_\beta$.
\end{definition}

Notice that $\mathcal L_\mathcal D X$ can be given the structure of a model category, where weak equivalences and cofibrations are given levelwise, assuming that all the model categories in the diagram are sufficiently nice, e.g., cofibrantly generated.  On the other hand, $\mathcal Lim_\mathcal DX$ does not have the structure of a model category, since the weak equivalence requirement is not preserved by general limits and colimits.

Throughout this paper we do not actually need the homotopy limit $\mathcal Lim_\mathcal D X$ to have a model structure.  While it is not entirely satisfactory to have the homotopy limit of model categories not itself a model category, for our purposes we can just as easily think of it as a more general kind of homotopy theory.  In some cases, we can find a localization of the more general model structure so that the fibrant-cofibrant objects do have the maps weak equivalences, but for any given example it is difficult to verify whether this process can be done.

We first assume that all model categories in question are \emph{combinatorial} in that they are cofibrantly generated and locally presentable \cite[2.1]{duggercomb}.

\begin{theorem}
Let $\mathcal L_\mathcal D X$ be the lax homotopy limit of a $\mathcal D$-diagram $X$ of combinatorial model categories $\mathcal M_\alpha$, and assume that $\mathcal L_\mathcal D X$ has the structure of a right proper model category.  Then there exists a right Bousfield localization of $\mathcal L_\mathcal D X$ whose cofibrant objects $(x_\alpha, u_{\alpha, \beta}^\theta)$ have all $x_\alpha$ cofibrant and all maps $u_{\alpha, \beta}^\theta$ weak equivalences in $\mathcal M_\beta$.
\end{theorem}

\begin{proof}
Since the model categories $\mathcal M_\alpha$ are combinatorial, in particularly locally presentable, we can find, for each $\alpha$ a set $\mathcal A_\alpha$ of objects of $\mathcal M_\alpha$ which generates all of $\mathcal M_\alpha$ by $\lambda$-filtered colimits for some sufficiently large regular cardinal $\lambda$ \cite[2.2]{duggercomb}.  Further assume that all objects of each $\mathcal A_\alpha$ are cofibrant in $\mathcal M_\alpha$, a condition which can always be satisfied by passing to a presentation for $\mathcal M_\alpha$ as given by Dugger \cite{dugger}.

Given $a_\alpha \in \mathcal A_\alpha$ and a left Quillen functor $F_{\alpha, \beta}^\theta$, consider the class of all objects $x_\beta$ of $\mathcal M_\beta$ equipped with weak equivalences
\[ F_{\alpha, \beta}^\theta \rightarrow x_\beta. \]  Choose one cofibrant representative $x_\beta$ of this set, in $\mathcal A_\beta$ if possible.  If $x_\beta \notin \mathcal A_\beta$, then we want to add it to our generating set for $\mathcal M_\beta$.

Repeating this process for all $\alpha$ and all $a_\alpha \in \mathcal A_\alpha$, we obtain potentially larger sets $\mathcal A^{(1)}_\alpha$ of cofibrant objects of $\mathcal M_\alpha$.

However, we do not know that the images under $F_{\alpha, \beta}^\theta$ of the elements $x_\alpha \in \mathcal A^{(1)}_\alpha$ which were not in $\mathcal A_\alpha$ have weakly equivalent objects in $\mathcal A^{(1)}_\beta$.  Thus, we repeat the process again to obtain sets $\mathcal A^{(2)}_\alpha$, and continuing, to get $\mathcal A^{(n)}_\alpha$.  Since we have inclusion maps
\[ \mathcal A_\alpha \rightarrow \mathcal A^{(1)}_\alpha \rightarrow \mathcal A^{(2)}_\alpha \rightarrow \cdots \] we take a colimit of this diagram to obtain a set $B_\alpha$ of cofibrant objects of $\mathcal M_\alpha$.

In $\mathcal L_\mathcal D X$, consider objects
\[ \{(x_\alpha, u_{\alpha, \beta}^\theta)_{\alpha, \beta, \theta} \mid x_\alpha \in \mathcal B_\alpha, U_{\alpha, \beta}^\theta \text{ weak equivalences in } \mathcal M_\beta\}. \]  We can obtain all objects of the homotopy limit as filtered colimits of this set, since filtered colimits preserves these weak equivalences \cite[7.3]{duggercomb}.  We take a right Bousfield localization of $\mathcal L_\mathcal D X$ with respect to this generating set and denote the resulting model category $\mathcal L$.

Recall that the class of cofibrant objects of $\mathcal L$ is the smallest class of cofibrant objects of $\mathcal L_\mathcal D X$ containing this set and closed under weak equivalences \cite[5.1.5, 5.1.6]{hirsch}.  Hence, to establish that this model structure is the one we want, we need to show that homotopy colimits of objects in our set still have the maps $u_{\alpha, \beta}^\theta$ weak equivalences.  But, any such homotopy colimit has the form
\[ (\hocolim_\gamma (x_\alpha)^\gamma, \hocolim_\gamma (u_{\alpha, \beta}^\theta)^\gamma) \] and each $\hocolim (x_\alpha)^\gamma$ is still cofibrant.  It follows then that each $\hocolim (u_{\alpha, \beta}^\theta)^\gamma$ is still a weak equivalence \cite[19.4.2]{hirsch}.
\end{proof}

Of course, the difficulty in using this theorem lies in the difficulty in establishing that the model category $\mathcal L_\mathcal D X$ is right proper.  We can weaken this condition, using a remark of Hirschhorn \cite[5.1.2]{hirsch} or use the structure, investigated by Barwick, which is retained after taking a right Bousfield localization of a model category which is not necessarily right proper \cite{barwick}.  In fact, Barwick considers a very similar construction to the one given here.  In practice, when the conditions of this theorem cannot be verified, we can still use the original levelwise model structure on $\mathcal L_\mathcal D X$ and simply restrict to the appropriate subcategory when we want to require $u$ and $v$ to be weak equivalences.

We now consider some simple examples of homotopy limits of model categories.

\begin{example}
Let $\mathcal D$ be the category $d_1 \rightarrow d_2$.  Then, a corresponding diagram of model categories has the form $F \colon \mathcal M_1 \rightarrow \mathcal M_2$.  The homotopy limit has objects $(x_1, x_2; u)$ where $x_i$ is an object of $\mathcal M_i$ for $i=1,2$, and $u \colon F(x_1) \rightarrow x_2$ is a weak equivalence in $\mathcal M_2$.  Thus, the homotopy limit is equivalent to the weak essential image of $\mathcal M_1$ in $\mathcal M_2$, i.e., the subcategory of $\mathcal M_2$ whose objects are weakly equivalent to objects in the image of $F$.
\end{example}

\begin{example}
We now consider an equalizer diagram $d_1 \rightrightarrows d_2$.  Our diagram of model categories then looks like
\[ \xymatrix@1{\mathcal M_1 \ar@<.5ex>[r]^{F_1} \ar@<-.5ex>[r]_{F_2} & \mathcal M_2.} \]  The homotopy limit has objects $(x_1, x_2; u_1, u_2)$ with $x_i$ an object of $\mathcal M_i$ for $i=1,2$ and $u_1, u_2$ weak equivalences
\[ \xymatrix@1{F_1(x_1) \ar[r]^-{u_1} & x_2 & F_2(x_1). \ar[l]_-{u_2}} \]  Thus, the equalizer of model categories is equivalent to the subcategory of $\mathcal M_2$ which is in the weak essential image of both $F_1$ and $F_2$.
\end{example}

\begin{example} \label{neweq}
Our definition of homotopy limit specializes to the definition of homotopy fiber product as given in \cite{fiberprod}.  Namely, given a diagram
\[ \xymatrix{& \mathcal M_2 \ar[d]^{F_2} \\
\mathcal M_1 \ar[r]^{F_1} & \mathcal M_3} \] its homotopy fiber product consists of 5-tuples $(x_1, x_2, x_3; u,v)$ with $x_i$ an object of $\mathcal M_i$ and $u$ and $v$ weak equivalences giving
\[ \xymatrix@1{F_1(x_1) \ar[r]^u & x_3 & F_2(x_2) \ar[l]_v. } \]

On particular example is given by applying this construction to the diagram
\[ \xymatrix{ & \mathcal M_1 \ar[d]^{(F_2,F_1)} \\
\mathcal M_1 \ar[r]^-{(F_1,F_2)} & \mathcal M_2 \times \mathcal M_2} \] giving an alternate model for the homotopy equalizer than the one given in the previous example.
\end{example}

\section{An application to a Postnikov tower of categories of topological spaces} \label{ex}

As an example of homotopy limits of model categories, we show in this section that the category of topological spaces is equivalent to the homotopy limit of model categories of topological spaces with nontrivial homotopy groups only below dimension $n$, with $n \geq 0$.

Consider the model category $\Top$ of topological spaces.  For each natural number $n$, define a left Bousfield localization of $\Top$, denoted $\Top_{\leq n}$, with respect to the set of maps
\[ \{S^k \rightarrow \ast \mid k>n\}. \]  The weak equivalences in $\Top_{\leq n}$ are the $n$-equivalences, or maps $X \rightarrow Y$ such that the induced maps
\[ \pi_i(X) \rightarrow \pi_i(Y) \] are weak homotopy equivalences for all $i \leq n$.  These model structures form a diagram of left Quillen functors
\[ \cdots \rightarrow \Top_{\leq 3} \rightarrow \Top_{\leq 2} \rightarrow \Top_{\leq 1} \rightarrow \Top_{\leq 0}. \]

Using the definition from the previous section, this diagram has homotopy limit consisting of objects
\[ \cdots \rightarrow X_3 \rightarrow X_2 \rightarrow X_1 \rightarrow X_0 \] in which each map $X_{n+1} \rightarrow X_n$ is an $n$-equivalence and morphisms given by $n$-equivalences $X_n \rightarrow Y_n$ for all $n \geq 0$ such that the resulting diagram commutes.

There is a functor $\Top \rightarrow \mathcal Lim_n \Top_{\leq n}$ where a space $X$ is sent to the constant sequence on $X$.  In the other direction, there is a functor sending a diagram of spaces
\[ \cdots \rightarrow X_3 \rightarrow X_2 \rightarrow X_1 \rightarrow X_0 \] to its homotopy limit $\holim_n X_n$ in $\Top$.

Starting with a space $X$, taking the constant diagram followed by its homotopy limit results in the space $X$.  Composing in the other direction, we need to prove that a diagram of spaces
\[ \cdots \rightarrow X_3 \rightarrow X_2 \rightarrow X_1 \rightarrow X_0 \] is equivalent to the constant diagram given by $\holim_n X_n$.  In other words, we want to show that the homotopy limit of this diagram is $n$-equivalent to the space $X_n$ for all $n \geq 0$.  But this fact follows after applying \cite[19.6.13]{hirsch}.

\section{Relationship with homotopy limits of complete Segal spaces}

Let $X$ be a $\mathcal D$-diagram of left Quillen functors between model categories, and let $\mathcal Lim_\mathcal D X$ be its homotopy limit.  Let $L_CX$ denote the diagram of complete Segal spaces obtained by applying the functor $L_C$ to $X$.  Then there exists a natural map $L_C \mathcal Lim_\mathcal D X \rightarrow \holim_\mathcal D L_C \mathcal M_\alpha$, where $\alpha$ ranges over the objects of the category $\mathcal D$, as follows.  Evaluation at $\alpha$ defines a functor $\mathcal Lim_\mathcal D \rightarrow \mathcal M_\alpha$ for every object $\alpha$ in $\mathcal D$, inducing a map of complete Segal spaces $L_C \rightarrow \mathcal Lim_\mathcal D \rightarrow L_C \mathcal M_\alpha$.  The universal property of limits gives the existence of a map
\[ L_C \mathcal Lim_\mathcal D \rightarrow \lim_\mathcal D L_C \mathcal M_\alpha. \]  Composing with the natural map
\[ \lim_\mathcal D L_C \mathcal M_\alpha \rightarrow \holim_\mathcal D L_C \mathcal M_\alpha \] \cite[19.2.10]{hirsch} gives the desired map.  Since we are working with complete Segal spaces, this map can be understood by levelwise calculations on simplicial sets.

\begin{theorem}
The map $L_C \mathcal Lim_\mathcal D X \rightarrow \holim_\mathcal D L_C \mathcal M_\alpha$ is a weak equivalence of complete Segal spaces.
\end{theorem}

To prove this theorem, we would like to be able to use Theorem \ref{baut} which characterizes the complete Segal spaces that result from applying the functor $L_C$ to a model category.  However, this theorem only gives the homotopy type of each space in the simplicial diagram, not an explicit description of the precise spaces we obtain.  We can justify using this description in terms of the homotopy type using an argument such as the one found in \cite[\S 4]{fiberprod} with still more details in \cite[\S 7]{css}, where Theorem \ref{baut} is proved.

\begin{proof}
We begin by considering the case where the diagram $X$ is connected.  Using our characterization of the homotopy type of the complete Segal space $L_C \mathcal M_\alpha$, we can assume that at level zero we get
\[ (\holim_\mathcal D L_C \mathcal M_\alpha)_0 \simeq \holim_\mathcal D \left( \coprod_{\langle x_\alpha \rangle} B \Aut^h(x_\alpha) \right). \]
Similarly, we can see that
\[ (L_C \mathcal Lim_\mathcal DX)_0 \simeq \coprod_{\langle (x_\alpha, u^\theta_{\alpha, \beta})_{\alpha, \beta, \theta} \rangle} B \Aut^h(x_\alpha, u^\theta_{\alpha, \beta}) \] which, using properties of the classifying space functor $B$, is equivalent to
\[ B\left( \coprod_{\langle (x_\alpha, u^\theta_{\alpha, \beta}) \rangle} \Aut^h(x_\alpha, u^\theta_{\alpha, \beta}) \right). \]

Let us first consider the simplicial monoids $\Aut^h(x_\alpha, u^\theta_{\alpha, \beta})$.  Its elements are families $(a_\alpha)_\alpha$ where each $a_\alpha \in \Aut^h(x_\alpha)$ such that each square diagram
\[ \xymatrix{F^\theta_{\alpha, \beta}(x_\alpha) \ar[r]^-{u^\theta_{\alpha, \beta}} \ar[d] _{F_{\alpha, \beta}^\theta(a_\alpha)} & x_\beta \ar[d]^{a_\beta} \\
F_{\alpha, \beta}^\theta(x_\alpha) \ar[r]^-{u^\theta_{\alpha, \beta}} & x_\beta} \] commutes.

We'd like to show that $(\holim_\alpha L_c \mathcal M_\alpha)_0$ is given by the same collection of commutative diagrams.  The process of taking coproducts commutes with homotopy limits, since we have assumed that $\mathcal D$ is connected, and the functor $B$ does as well \cite[19.4.5]{hirsch}, so we obtain equivalences
\[ \begin{aligned}
(\holim_\alpha L_c \mathcal M_\alpha)_0 & \simeq \holim_\alpha \left( \coprod_{\langle x_\alpha \rangle} B\Aut^h(x_\alpha) \right) \\
& \simeq B \left( \holim_\alpha \left( \coprod_{\langle x_\alpha \rangle} \Aut^h(x_\alpha) \right) \right) \\
& \simeq B \left( \coprod_{\langle x_\alpha \rangle}  \holim_\alpha \Aut^h(x_\alpha) \right).
\end{aligned} \]
Thus, it suffices to show that $\holim_\alpha \Aut^h(x_\alpha)$ consists of diagrams as given above.

This fact is more difficult than it was in the special case of homotopy pullbacks, since it is less common to think of them in this way.  However, using Bousfield and Kan's definition of the homotopy limit as $\holim(X) = \Hom(\mathcal D/-, X)$ for a diagram $X$ indexed by a category $\mathcal D$ \cite{bk}, one can check that the diagram above corresponds exactly to the one that we get by taking the category under $\mathcal D$ as used in this definition.

A similar argument can be used in dimension 1 (and subsequently in higher dimensions), using that
\[ (\holim_\mathcal D L_C \mathcal M_\alpha)_1 \simeq \holim_\mathcal D \left( \coprod_{\langle \alpha \colon x \rightarrow y \rangle} B\Aut^h(\alpha) \right) \] and
\[ (L_C \mathcal Lim_\mathcal D X)_1 \simeq \coprod_{\langle f \rangle} B\Aut^h(f) \] where
\[ f=(f_\alpha) \colon ((x_\alpha), (u_{\alpha, \beta}^\theta)) \rightarrow ((y_\alpha), (v_{\alpha, \beta}^\theta)) \] is a morphism in $\mathcal Lim_\mathcal DX$.

For the general case, where the diagram $\mathcal D$ has multiple connected components, we can write the map $L_C \lim_\mathcal D X \rightarrow \holim_\mathcal D L_C \mathcal M_\alpha$ as 
\[ L_C \left(\prod_{j \in J} \lim_{\mathcal D_j} X \right) \rightarrow \prod_{j \in J} \holim_{\mathcal D_j} L_C \mathcal M_\alpha \] where the set $J$ indexes the connected components $\mathcal D_j$ of $\mathcal D$.  Since the functor $L_C$ commutes with products, the result follows from the previous case.
\end{proof}

\section{Complete Segal spaces arising from the lax homotopy limit construction}

In \cite{fiberprod}, we give a concise description of the complete Segal space corresponding to the lax homotopy fiber product construction.  We would like to have an analogous description for more general homotopy limits, but we need a restriction to get one in the same way.  In particular, we need to assume that our diagram category $\mathcal D$ has a terminal object.  Recall that by $\nerve(\mathcal D)^t$ we denote the simplicial space whose space of $n$-simplices is the discrete space of $n$-simplices of the simplicial set $\nerve(\mathcal D)$.

\begin{theorem}
Let $\mathcal L_\mathcal D X$ denote the lax homotopy limit of a $\mathcal D$-shaped diagram $X$ of model categories, and suppose that $\mathcal D$ has a terminal object $\omega$.  Then the complete Segal space $L_C \mathcal L_\mathcal D X$ is weakly equivalent to the pullback in the diagram
\[ \xymatrix{L_C \mathcal L_\mathcal D X \ar[r] \ar[d] & L_C (\mathcal M_\omega)^{\nerve(\mathcal D)^t} \ar[d] \\
\prod_{\alpha \in \ob(\mathcal D)} L_C \mathcal M_\alpha \ar[r]^{\prod F_{\alpha, \omega}^\theta} & \prod_{\alpha \in \ob(\mathcal D)} L_C \mathcal M_\omega} \]  where in the bottom horizontal arrow $\theta$ is taken in each case to be the unique map in $\mathcal D$ from $\alpha$ to the terminal object $\omega$.
\end{theorem}

\begin{proof}
We would like to apply the strategy used in the proof of the analogous result for pullback diagrams \cite[5.4]{fiberprod}.  Let $I$ denote the nerve of the category $(\cdotp \rightarrow \cdotp)$.  For every morphism $d_{\alpha \beta}^\theta$ in $\mathcal D$, we get a diagram
\[ \xymatrix{ & \Map(I^t, L_C \mathcal M_\beta) \ar[d] \\
L_C \mathcal M_\alpha= \Map(\Delta [0]^t, L_C \mathcal M_\alpha) \ar[r] & \Map(\Delta[0]^t, L_C \mathcal M_\beta)= L_C \mathcal M_\beta} \] where the horizontal map is induced by $F_{\alpha \beta}^\theta$ and the vertical map is induced by the inclusion of the terminal object $\Delta[0] \rightarrow I$.
However, assembling all such diagrams together requires a terminal object.  When such a terminal object $\omega$ exists in $\mathcal D$, we get the pullback of the diagram
\[ \xymatrix{ & \Map((\nerve(\mathcal D))^t, L_C \mathcal M_\omega) \ar[d] \\
\prod_{\alpha \in \ob(\mathcal D)} \Map(\Delta [0]^t, L_C \mathcal M_\alpha) \ar[r] & \prod_{\alpha \in \ob(\mathcal D)} \Map(\Delta[0]^t, L_C \mathcal M_\omega).} \]  The horizontal map is given by
\[ \prod_{\alpha \in \ob(\mathcal D)} \Map(\Delta[0]^t, F^\theta_{\alpha \omega}, \] with $\theta$ the unique map $\alpha \rightarrow \omega$ in $\mathcal D$, and the vertical map is given by
\[ \prod_{\alpha \in \ob(\mathcal D)} \Map(i_\alpha, L_C \mathcal M_\omega) \] where $i_\alpha \colon \Delta[0]^t \rightarrow (\nerve(\mathcal D))^t$ is given by inclusion at $\alpha$.
Then the mapping spaces in the bottom row are equivalent to the ones given in the statement of the theorem.  The compatibility of the maps $F_{\alpha \omega}^\theta$ in the diagram $X$, together with the analogous compatibility in $\nerve(\mathcal D)$, guarantees the necessary compatibility in the definition of lax homotopy limit.
\end{proof}

\end{document}